\documentclass{amsart}

\usepackage{color}
\usepackage{amsthm,amssymb,verbatim}
\usepackage{graphicx}
\usepackage{enumerate}

\newcommand{\TC}{\operatorname{TC}}

 \newcommand{\slope}{\operatorname{slope}}

 \newcommand{\RR}{\mathbf{R}}  
 \newcommand{\BB}{\mathbf{B}}  
 \newcommand{\dist}{\operatorname{dist}}

 \newcommand{\eps}{\epsilon}
 \newcommand{\Tan}{\operatorname{Tan}}

\newcommand{\vv}{\bold v}

\newcommand{\ee}{\mathbf{e}}

\input epsf
\def\begfig {
\begin{figure}
\small }
\def\endfig {
\normalsize
\end{figure}
}

    \newtheorem{claim}{Claim}

\swapnumbers

    \newtheorem{theorem}    {Theorem}       [section]
    \newtheorem{lemma}      [theorem]       {Lemma}
    \newtheorem{corollary}  [theorem]     {Corollary}

    \newtheorem*{theorem*}{Theorem}
    \newtheorem*{corollary*}{Corollary}

    \theoremstyle{definition}

    \theoremstyle{definition}
    \newtheorem{remark}   [theorem]       {Remark}

\usepackage{hyperref} 
\usepackage[alphabetic, msc-links, backrefs]{amsrefs}

\begin{document}

\title[Embedded Minimal Surfaces]{On the Compactness Theorem for Embedded Minimal Surfaces in $3$-manifolds with
   Locally Bounded Area and Genus}
\author{Brian White}
\address{Department of Mathematics\\ Stanford University\\ Stanford, CA 94305}
\thanks{The research was supported by NSF
   grants~DMS--1105330 and DMS--1404282}
\email{white@math.stanford.edu}
\date{March 6, 2015.}
\begin{abstract}
Given a sequence of properly embedded minimal surfaces in a $3$-manifold with
local bounds on area and genus, we prove subsequential convergence, smooth
away from a discrete set, to a smooth embedded limit surface, possibly with multiplicity,
and we analyze what happens when one blows up the surfaces near a point 
where the convergence is not smooth.
\end{abstract}

\maketitle

\renewcommand{\thesubsection}{\thetheorem}

\section{Introduction}

In this paper, we prove several results, most of which can be summarized as follows:

\begin{theorem}\label{summary-theorem}
Let $\Omega$ be an open subset of a Riemannian $3$-manifold.
Let $g_i$ be a sequence of smooth Riemannian metrics on $\Omega$ converging smoothly
to a Riemannian metric $g$.
Let $M_i\subset \Omega$ be a sequence of properly embedded surfaces 
such that $M_i$ is minimal with respect to $g_i$.
Suppose also that  the area and the genus of $M_i$ are uniformly bounded on compact subsets of $\Omega$.
Then (after passing to a subsequence) the $M_i$ converge to a  smooth, properly embedded
$g$-minimal surface $M$.  For each connected component $\Sigma$ of $M$,
either
\begin{enumerate}
\item the convergence to $\Sigma$ is smooth with multiplicity one, or
\item the convergence is smooth (with some multiplicity $>1$) away from a discrete set $S$.
\end{enumerate}
In the second case, if $\Sigma$ is two-sided, then it must be stable.

Now suppose that $\Omega$ is an open subset of $\RR^3$.  (The metric $g$ need not be flat.)
If $p_i\in M_i$ converges to $p\in M$, then (after passing to a further subsequence)
either
\[
    \Tan(M_i,p_i)\to \Tan(M,p),
\]
or there exists constants $\lambda_i>0$ tending to $\infty$ such that the surfaces
\[
    \lambda_i(M_i - p_i)
\]
converge smoothly and with multiplicity $1$
 to a non-flat, complete, properly embedded minimal surface $M'\subset \RR^3$ of finite total 
curvature with ends parallel to $\Tan(M,p)$.
\end{theorem}

Compactness theorems similar to Theorem~\ref{summary-theorem} were proved
in~\cite{choi-schoen}, \cite{anderson}, and~\cite{white-curvature-estimates}, and the proof of 
much of Theorem~\ref{summary-theorem} is very similar to the proofs in~\cite{choi-schoen} and 
\cite{white-curvature-estimates}.
However, \cite{choi-schoen} and~\cite{anderson} have hypotheses that are considerably more restrictive than
the hypotheses in Theorem~\ref{summary-theorem}. 
Also, Theorem~\ref{summary-theorem} has some very useful conclusions that are
not stated in~\cite{choi-schoen}, \cite{anderson}, or \cite{white-curvature-estimates}.
In particular, the conclusion that the ends of $M'$ are parallel to $\Tan(M,p)$
seems to be entirely new.   That conclusion is used in an essential way in the
recent proof~\cite{hoffman-traizet-white-2} of existence of helicoidal surfaces of arbitrary genus in $\RR^3$.  
The bulk of this paper (Sections~\ref{annulus-section} and~\ref{proof-section})
is devoted to proving that conclusion.

If one drops the assumption that the $M_i$ have locally bounded areas, the behavior
becomes considerably more complicated.
For example, even for simply connected $M_i$ in an open subset of $\RR^3$,
the curvatures of the $M_i$ can blow up on arbitrary $C^{1,1}$ curves \cite{meeks-weber-bending}
or on 
arbitrary closed subsets (such as Cantor sets) of a line~\cite{hoffman-white-sequences},~\cite{Kleene}.
See~\cite{ColdingMinicozzi-1}, \cite{ColdingMinicozzi-2}, \cite{ColdingMinicozzi-3}, \cite{ColdingMinicozzi-4},
and~\cite{meeks-regular}
for very powerful theorems analyzing the behavior of such sequences.
Based on those works, \cite{white-C1}*{corollary~3 and theorem~4}
 formulates a compactness theorem somewhat analogous
to the Compactness Theorem~\ref{summary-theorem} in this paper.

It would be very interesting to analyze what happens if one assumes local bounds on area but
 not on genus.
By passing to a subsequence, one can get weak convergence to a stationary integral varifold $V$.
The limit varifold has associated to it a flat chain mod $2$, and that  flat chain has
no boundary in the open set~\cite{white-currents}.  Thus, for example, the varifold
cannot have soapfilm-like triple junctions.  In fact, \cite{white-currents}
also proves the slightly stronger statement that if the original surfaces are orientable, then
 there is an integral current $T$ with no boundary
in the open set such that $T$ and $V$ determine the same flat chain mod $2$.
(The results in~\cite{white-currents} hold for arbitrary dimension and codimension.)
Nothing else seems to be known
about the class of stationary integral varifolds $V$ that arise as such a limit.

\section{The Main Theorems}\label{main-section}

If $M$ is a surface in a Riemannian $3$-manifold, we let $\TC(M)$ denote the
total curvature of $M$:
\[
   \TC(M) = \frac12\int_M(\kappa_1^2 + \kappa_2^2)\,dA,
\]
where $\kappa_1$ and $\kappa_2$ are the principal curvatures of $M$.

\begin{theorem}[Compactness Theorem]\label{compactness-theorem}
Let $\Omega$ be an open subset of smooth $3$-manifold.
Let $g_i$ be a sequence of smooth Riemannian metrics on $\Omega$ converging smoothly
to a Riemannian metric $g$.
Let $M_i\subset \Omega$ be a sequence of properly embedded surfaces 
such that $M_i$ is minimal with respect to $g_i$.
Suppose also that  the area and the genus of $M_i$ are bounded independently of $i$ on compact
subsets of $\Omega$.

Then the total curvatures of the $M_i$ are also uniformly bounded on compact subsets of $\Omega$.
After passing to a subsequence, the $M_i$ converge to a  smooth, properly embedded,
$g$-minimal surface $M$, and the convergence is smooth away from a discrete set $S$.
  For each connected component $\Sigma$ of $M$,
either
\begin{enumerate}
\item the convergence to $\Sigma$ is smooth everywhere with multiplicity $1$, or
\item the convergence to $\Sigma$ is smooth with some multiplicity $>1$ away from $\Sigma\cap S$.
 In this case, if $\Sigma$ is two-sided, then it must be stable.
\end{enumerate}
If the total curvatures of the $M_i$ are bounded by $\beta$, then $S$ has at most $\beta/(4\pi)$ points.
\end{theorem}

\begin{proof}
See \cite{ilmanen-singularities}*{Theorem~3} for a proof that the total curvatures of the $M_i$ are uniformly bounded on compact
subsets of $\Omega$.

If the supremum of the total curvatures of the $M_i$ is less than $4\pi$, then
we get smooth, subsequential convergence (possibly with multiplicity) everywhere.
(This follows from the curvature estimate~\cite{white-utah}*{Theorem~24} 
or~\cite{white-curvature-estimates}*{p. 247-248}.)

It follows (after passing to a subsequence) that there is a discrete set $S$ such that
the $M_i$ converge smoothly to $M$ on compact subsets of $\Omega\setminus S$, where
$M\setminus S$ is a smooth minimal surface properly immersed in $\Omega\setminus S$.
Furthermore, if $W$ is an open subset of $\Omega$, then the number of points in $S\cap W$
is at most
\[
   \frac1{4\pi} \limsup_i \TC(M_i\cap W).
\]
See~\cite{white-curvature-estimates} or~\cite{white-utah}*{theorem~25} for details.

Since the $M_i$ are embedded, $M\setminus S$ has no transverse self-intersections.
Hence $M\setminus S$ is smooth and embedded, possibly with multiplicity.  It also follows that the points in $S$
are removable singularities of $M$.  
(This removal of singularities theorem can be proved in a variety of ways.  
See the appendix for one proof.)
In other words, $M$ is a smooth, embedded surface, possibly with multiplicity $>1$.

For simplicity, let us assume $M$ has just one connected component.

If $M$ has multiplicity $1$, then the convergence of $M_i$ to $M$ is smooth everywhere
by Allard's Regularity Theorem \cite{allard-first-variation}*{\S8} or \cite{simon-book}*{\S23--\S24}, 
or by the easy version of the 
Allard Regularity Theorem  in \cite{white-local-regularity}*{Theorem~1.1} or \cite{white-utah}.
(The proof in \cite{white-local-regularity} is for compact surfaces, but
that proof can easily be modified to handle proper, non-compact surfaces.)

Thus suppose that the $M_i$ converge to $M$ with multiplicity $k>1$,
and suppose that $M$ is two-sided.  Since the convergence is smooth on compact
subsets of $\Omega\setminus S$, 
we can (away from $S\cup \partial \Omega$)
express $M_i$ as the union of $k$ disjoint, normal graphs over $\Sigma$. Since $M_i$ is embedded,
the functions can be ordered.   
 Let $\phi_i$ be the difference of the largest and the smallest functions.
Let $p$ be a point in $\Sigma\setminus S$.   
By standard PDE, $\phi_i$ satisfies a  second-order linear elliptic equation.  By the Harnack inequality
and the Schauder estimates, the functions $\phi_n/ |\phi_n(p)|$ converge smoothly (after
passing to a subsequence) to a positive jacobi field $\phi$ on $\Sigma\setminus S$.
By~\cite{fischer-colbrie-schoen}*{Theorem~1}, existence of such a $\phi$ implies that 
  $\Sigma\setminus S$ is stable.  
A standard cut-off argument (cf. Corollary~\ref{cutoff-corollary})
 shows that  $\Sigma$ and $\Sigma\setminus S$ have
the same jacobi eigenvalues.  Thus $\Sigma$ is stable.

This completes the proof of Theorem~\ref{compactness-theorem}.
\end{proof}

The remaining results are local, so we can assume that $\Omega$ is an open subset
of $\RR^3$. (Of course the metrics $g_i$ and $g$ need not be flat.)
We let $\lambda(M-p)$ denote the result of translating $M$ by $-p$ and
then dilating by $\lambda$.

\begin{theorem}[Blow-up Theorem]\label{blow-up-theorem}
Suppose in the Compactness Theorem~\ref{compactness-theorem} that $\Omega$ is an open subset of $\RR^3$.
Suppose $p_i\in M_i$ converges to $p\in M$, and that $\lambda_i\to\infty$.

Then, after passing to a subsequence, the surfaces 
\[
  M_i':=\lambda_i(M_i-p_i)
\]
converge smoothly away from a finite set $Q$ to a complete, properly embedded, $g(p)$-minimal
surface $M'$ of finite total curvature. 

Furthermore, $M'$ must be one of the following:
\begin{enumerate}
\item a multiplicity $1$ plane,
\item a complete, non-flat, properly embedded surface of finite total curvature, with multiplicity $1$, or
\item the union of two or more (counting multiplicity) parallel planes.
\end{enumerate}
In cases (1) and (2), the convergence of $M_n'$ to $M'$ is smooth everywhere.
\end{theorem}

\begin{proof}
The monotonicity formula implies that the areas of the $M_i'$ are uniformly
bounded on compact sets.
Thus the subsequential convergence to a complete, smooth, properly embedded $g(p)$-minimal surface
$M'$ of finite total curvature and the finiteness of the set $Q$ follow immediately from
the Compactness Theorem~\ref{compactness-theorem}. 

Suppose that $M'$ is not the union of one or more parallel planes.
By the Strong Halfspace Theorem~\cite{hoffman-meeks-halfspace}*{Theorem~2},
    $M'$ is connected.
Since $M'$ is not a plane, it is 
unstable (by~\cite{fischer-colbrie-schoen} or \cite{do-carmo-peng}).
Thus by the Compactness Theorem~\ref{compactness-theorem}, $M'$ has multiplicity $1$.

The smooth convergence everywhere in cases (1) and (2) follows 
from the Compactness Theorem~\ref{compactness-theorem}.
\end{proof}

\begin{theorem}[No\,-Tilt Theorem]\label{no-tilt-theorem}
In cases (2) and (3) of the Blow-up Theorem~\ref{blow-up-theorem}, 
the ends of $M'$ are parallel to $\Tan(M,p)$.
\end{theorem}

The proof will be given in Sections~\ref{annulus-section} and~\ref{proof-section}.

\begin{theorem}\label{nonflat-theorem}
Suppose, in the Compactness Theorem~\ref{compactness-theorem}, that $\Omega$ is an open subset of $\RR^3$.
Suppose that $p_i\in M_i$ converges to $p\in M$ and that $\Tan(M_i, p_i)$ does not
converge to $\Tan(M,p)$.
Then there exist $\lambda_i\to\infty$ such that, after a passing to a subsequence,
the surfaces
\[
  M_i':= \lambda_i (M_i - p_i)
\]
converge smoothly and with multiplicity $1$ to a complete, smooth, properly embedded, non-flat, $g(p)$-minimal surface $M'$
of finite total curvature.  Furthermore, the ends of $M'$ must be parallel to $\Tan(M,p)$.
\end{theorem}

\begin{proof}
By passing to a subsequence, we can assume that $\Tan(M_i, p_i)$ converges to a plane
$P$ not equal to $\Tan(M,p)$.
By the Blow-up Theorem~\ref{blow-up-theorem}, it suffices to show that we can choose $\lambda_i\to \infty$
such that a subsequential limit $M'$ of the surfaces
\[
  M_i':=\lambda_i(M_i-p_i)
\]
is not the union of one or more planes.

Let $r_i$ be the infimum of the numbers
$r>0$ such that
\[
    M_i\cap \BB(p_i,r)
\]
contains a point at which the principal curvatures are $\ge 1/r$.
The hypothesis on $\Tan(M_i, p_i)$ implies that the principal curvatures of the $M_i$ are
not bounded on any neighborhood of $p$, and hence that
\begin{equation*}
  \liminf_i r_i = 0.
\end{equation*}
By passing to a subsequence, we can assume that $r_i\to 0$.  Now let $\lambda_i=1/r_i$.

The Blow-up Theorem~\ref{blow-up-theorem} implies that, after passing to a further subsequence, the surfaces
\[
   M_i' := \lambda_i (M_i - p_i)
\]
converge to a limit surface $M'$.  

By the choice of $\lambda_i$, the surface $M_i'$ converges to $M'$ smoothly in the open
ball of radius $1$ about $0$.  Hence 
\[
\Tan(M',0) = \lim_i \Tan(M_i',0) = \lim_i \Tan(M_i,p_i) = P \ne \Tan(M,p),
\]
so
\begin{equation}\label{eq:different-planes}
  \Tan(M',0) \ne \Tan(M,p).
\end{equation}

We claim that $M_i'$ converges to $M'$ smoothly everywhere.
For, if not, then by the Blow-up Theorem~\ref{blow-up-theorem}, 
$M'$ would be the union of two or more planes parallel to $\Tan(M,p)$.
But that contradicts~\eqref{eq:different-planes}.  

Thus $M_i'$ converges smoothly to $M'$.
Since (by choice of $\lambda_i$) each $M_i'\cap\partial \BB(0,1)$ contains
a point of $M_i'$ where the principal curvatures are $\ge 1$, the surface $M'$ is smooth
and non flat. 
The remaining conclusions follow immediately from the Blow-up Theorem~\ref{blow-up-theorem}.
(Since $M'$ is not flat, we are in case (2) of that theorem.)
\end{proof}

\section{The Annulus Lemma}\label{annulus-section}

The proof of the No\,-Tilt Theorem relies heavily on the following lemma, 
which describes the behavior of nearly flat minimal annuli as the inner radius tends to $0$:

\begin{lemma}[Annulus Lemma]\label{annulus-lemma}
Let $g_i$ be a sequence of Riemannian metrics on
the cylinder
\[
   C(R,a):= \{ (x,y,z)\in\RR^3:  x^2+y^2\le R^2, \, |z|\le a \}
\]
that converge smoothly to a Riemannian metric $g$. 
For $i=1,2, \dots$, suppose  that $M_i\subset C(R,a)$ is a $g_i$-minimal surface
that is  the graph of a function
\[
   u_i: A(r_i,R)\to (-a,a),
\]
where 
\[
   A(r_i,R)= \{p\in\RR^2: r_i\le |p|\le R\}
\]
and where the radii $r_i$ are positive numbers that converge to $0$.
Suppose also that
\[
L:=\sup_i \sup |Du_i| < \infty,
\]
and that
\[
\text{$u_i\to 0$ smoothly on $A(\eta,R)$ for every $\eta>0$}.
\]
Let $\lambda_i$ be a sequence of numbers tending to infinity
such that 
\begin{equation}\label{eq:r-prime-finite}
  r' = \lim \lambda_ir_i <\infty.
\end{equation}

Let $\vv_i=(0,0,v_i)$ be a point on the $z$-axis such that
\begin{equation}\label{eq:not-too-far}
  c:= \sup_i\frac{\dist(\vv_i, M_i)}{r_i} <\infty.
\end{equation}
Let $M_i' = \lambda_i(M_i-\vv_i)$,
so that $M_i'$ is the graph of the function 
\begin{align*}
&u_i': A(\lambda_ir_i,\lambda_iR)\to \RR, \\
&u_i'(q) =  \lambda_i( u_i(q/\lambda_i) - v_i).
\end{align*}
Then, after passing to a subsequence, 
the $u_i'$ converge uniformly on compact subsets
of $\RR^2$ to a function
\[
  u': A(r',\infty)\to\RR.
\]
The convergence is smooth on compact subsets
of $\sqrt{x^2+y^2}>r'$, and
\begin{equation}\label{eq:flattens}
   \lim_{|q|\to\infty} |Du'(q)|=0.
\end{equation}
\end{lemma}

\begin{proof}
Except for the last statement~\eqref{eq:flattens}, the lemma is
straightforward, as we now explain.
Note that the graphs of the $u_i'$ have slopes bounded by $L$.
Also,
\[
   \limsup_i \dist(0, M_i') \le cr' < \infty
\]
by~\eqref{eq:r-prime-finite} and~\eqref{eq:not-too-far}.
By the Arzela-Ascoli Theorem, after passing to a subsequence, the $u_i'$ 
converge uniformly on compact subsets
of $\RR^2$ to a limit function
\[
  u': A(r',\infty)\to \RR.
\]
By standard PDE or by minimal surface regularity theory, the convergence is smooth
on compact subsets of $\sqrt{x^2+y^2}>r'$.

It remains to prove the last assertion~\eqref{eq:flattens}.
We remark that if~\eqref{eq:flattens} holds
for one choice of $\vv_i$, then it also holds for any other choice $\vv_i^*$ 
(subject to the condition~\eqref{eq:not-too-far}).
This is because if $\lambda_i(M_i-\vv_i)$ converges to $M'$, then, after
passing to a further subsequence, the surfaces $\lambda_i(M_i-\vv_i^*)$ converge
to a limit $M^*$, and clearly $M^*$ is a vertical translate of $M'$.

Let $D$ be the horizontal disk of radius $R$ centered at the origin.
The smooth convergence $M_i\to D$ away from the origin implies that the
mean curvature with respect to $g$ of $D$ vanishes everywhere except possibly
at the origin. By continuity, it must also vanish at the origin.  That is, $D$ is a $g$-minimal 
surface.

By replacing $R$ by a sufficiently small $\hat{R}>0$ and $M_i$ by 
$M_i\cap C(\hat{R},a)$,
we can assume that the disk $D$ is strictly stable.

\begin{claim}\label{disks-minimal-claim}
 It suffices to prove the lemma under that assumptions that
the outer boundary of $M_i$ lies in the plane $z=0$, i.e., that $u_i(p)\equiv 0$ when $|p|=R$,
and that horizontal disks (i.e., disks of the form $z=\text{constant}$) are $g_i$-minimal
for every~$i$.
\end{claim}

\begin{proof}[Proof of Claim~\ref{disks-minimal-claim}]
By the implicit function theorem and the strict stability of $D$, that there exist $\eps>0$ and $\delta>0$ 
with the following property:
for all $|t|\le \delta$, there is a unique smooth function 
\[
  f^t: \{ \sqrt{x^2+y^2}  \le R \} \to \RR
\]
such that $\|f^t\|_{2,\alpha} < \eps$, such that 
\[
  \text{$f^t = t$ on the circle $\sqrt{x^2+y^2}=R$},
\]
 and such that the graph of $f^t$ is a strictly stable, $g$-minimal disk.
 Note that $f^t$ depends smoothly on $t$.  Note also that $\frac{\partial}{\partial t}f^t$,
 which may be regarded as a jacobi field on the graph of $f^t$, is equal to $1$ on the boundary
 and therefore is everywhere positive by stability.
 Thus the map
 \[
  F:(x,y,z)\mapsto (x,y, f^t(x,y))
 \]
 is a smooth diffeomorphism from the cylinder $C(R,\delta)$
onto to its image.

Similarly (and also by the implicit function theorem), for all sufficiently large $i$
and for all $|t|\le \delta$, there is a smooth smooth function
\[
   f_i^t: \{\sqrt{x^2+y^2}\le R\} \to \RR
\]
such that $\|f^t_i\|_{2,\alpha}< \eps$,
such that  
\[
   \text{$f_i^t = u_i + t$ on the circle $\sqrt{x^2+y^2}=R$},
\]
and such that the graph of $f^t_i$ is a strictly stable, $g_i$-minimal surface.
Furthermore, 
\[
F_i^t: (x,y,z) \mapsto (x,y, f^t_i(x,y))
\]
defines a smooth diffeomorphism from $C(R,a)$ to its image, and  $F_i$
converges smoothly to $F$ as $i\to\infty$. 
(All these statements are consequences of the implicit function theorem.)

Now let $M_i'$ and $g_i'$ be the pull-back of $M_i$ and $g_i$ under the diffeomorphism $F_i$.
Then $M_i'$ and $g_i'$ satisfy all the hypotheses of the lemma, and, in addition, 
horizontal disks are $g_i'$ minimal and the outer boundary of the annulus $M_i'$ is a horizontal
circle centered at the origin.
This completes the proof of Claim~\eqref{disks-minimal-claim}.
\end{proof}

From now on, we will use the assumptions listed in Claim~\ref{disks-minimal-claim}.
By making a suitable diffeomorphic perturbation (supported near the origin)
of the form
\[
   (x,y,z)\mapsto (\phi(x,y), \psi(z)),
\]
and having the origin as a fixed point, we can further assume that the metric $g(0)$ coincides with the Euclidean metric
at the origin:
\begin{equation}\label{eq:euclidean}
    g(0)(\ee_i,\ee_j) = \delta_{ij}.
\end{equation}

We now prove~\eqref{eq:flattens}, under the additional assumptions indicated by 
Claim~\ref{disks-minimal-claim},  and also assuming~\eqref{eq:euclidean}.

If the $u_i$ are identically zero, there is nothing to prove.
Thus by passing to a subsequence, we may assume, without loss of generality,
that
\[
   z_i := \max u_i >0.
\]
(The case $\min u_i<0$ is proved in exactly the same way.)
By the maximum principle (and by the assumptions described in  Claim~\ref{disks-minimal-claim}),
 the maximum is attained on the inner boundary circle of $A(r_i,R)$.   
 (Recall that $z\equiv 0$ on the outer boundary circle.)
 Thus
\begin{equation}\label{eq:max-inside}
     u_i \le z_i = \max_{|p|=r_i}u_i(p).
\end{equation}

As explained earlier, the validity of the lemma  does not depend on the 
choice of $\vv_i=(0,0,v_i)$, so we may choose $\vv_i=(0,0,z_i)$.
It follows that $M_i'$ lies in the halfspace $z\le 0$ for all $i$, and thus so does $M'$:
\[
  u'\le 0.
\]
By~\eqref{eq:euclidean}, the surface $M'$ is minimal with respect to the standard Euclidean metric.

There are now many ways to see that $M'$ is horizontal at infinity.
For example, the tangent cone at infinity to $M'$ is a multiplicity-one Lipschitz
graph and therefore is a plane (because its intersection with the unit $2$-sphere must
be a geodesic).  Since it lies in the halfspace $\{z\le 0\}$, the plane must be
horizontal.

\end{proof}

\section{Proof of the No\,-Tilt Theorem}\label{proof-section}
 
We now prove the No\,-Tilt Theorem~\ref{no-tilt-theorem}.
We may assume that $p_i=p=0$. (Otherwise replace $M_i$ and $M$ by $M_i-p_i$
and $M-p$, and similarly for the metrics $g_i$ and $g$.)   
By rotation, we may assume that $\Tan(M,0)$ is horizontal. 
Thus it suffices to prove the following special case of the No\,-Tilt Theorem:

\begin{theorem}\label{theorem-restated}
Let $\Omega$ be an open subset of $\RR^3$ and let $g_i$ be a sequence of smooth
Riemannian metrics on $\Omega$ that converge smoothly to a Riemannian metric $g$.
Suppose that $M_i$ and $M$ are smooth, properly embedded surfaces 
in $\Omega$  such that $M_i$ is $g_i$-minimal, $M$ is $g$-minimal, 
 and such that $M_i$ converges smoothly,
with some finite multiplicity, to $M$  away from a discrete set of points.
Suppose also that 
\[
   \sup_i\TC(M_i) <\infty.
\]

Suppose that the origin is contained in each of the $M_i$,
and suppose that $\Tan(M,0)$ is horizontal.
Let $\lambda_i$ be
a sequence of numbers tending to $\infty$, and suppose that the dilated surfaces
$\lambda_i M_i$ converge smoothly away from a finite set of points
to a limit surface $M^*$. 

Then either $M^*$ is a multiplicity one plane, or the ends of $M^*$ are all horizontal:
\[
  \lim_{|q|\to \infty} \slope(\Tan(M^*,q)) = 0.
\]
\end{theorem}

\begin{proof}
Let $m$ be the multiplicity of the convergence $M_i\to M$.
If $M$ is not connected,  the multiplicity could be different on different components of $M$.
In that case, we let $m$ be the multiplicity on the connected component
of $M$ containing the origin.

Let $N$ be an integer such that $\sup_i\TC(M_i\cap U)<N$ for some open set $U$ containing 
the origin.

We will prove Theorem~\ref{theorem-restated} by double induction on the multiplicity $m$ and on $N$.
Thus we may assume that the theorem is true for surfaces $M_i'\to M'$  
(satisfying the hypotheses of
the theorem) provided
\begin{enumerate}
\item $m'<m$, or
\item $m'=m$ and $\sup_i\TC(M_i'\cap U' ) < N-1$ for some neighborhood $U'$ of $0$,
\end{enumerate}
where $m'$ is the multiplicity of convergence of $M_i'$ to $M'$ on the connected component
$M'$ containing the origin.

Since the result is local, we can assume that $M$ is topologically a disk.  By composing with a diffeomorphism,
we can assume that $M$ is a horizontal disk centered at the origin.  By composing with another diffeomorphism,
we can assume that the metric $g$ agrees with the Euclidean metric at the origin.
By replacing $\Omega$ by a small open set of the form
\[
     \{(x,y,z)\in \RR^3: x^2+y^2< R^2, \, |z| < a\},
\]
we can assume that the $\overline{M_i}$ are smooth manifolds-with-boundary that converge smoothly
 to $\overline{M}$ away from the origin.

If the convergence of $M_i$ to $M$ is smooth everywhere, then the result is trivially true:
in that case, every subsequence of $\lambda_iM_i$ has a further subsequence that converges smoothly
to the union of one or more horizontal planes.

Thus we may assume that the convergence is not smooth.
It follows that there is a sequence of points $p_i\in M_i$
converging to the origin such that $\Tan(M_i,p_i)$ does not converge to a horizontal plane.
By passing to a subsequence, we may assume that 
\begin{equation}\label{eq:steep}
     \slope(\Tan(M_i,p_i)) > L \ge 0
\end{equation}
for some $L>0$ and for all $i$.

Consider the set $S_i$ of points $q=(x,y,z)$ in $M_i$
such that 
\begin{align*}
&\text{$\slope(\Tan(M_i,q))\ge L$, or} \\
& |z|\ge L\sqrt{x^2+y^2}.
\end{align*}
Let 
\[
r_i = \max \{\sqrt{x^2+y^2}: (x,y,z)\in S_i\}.
\]
Note that $r_i>0$ by~\eqref{eq:steep} and that $r_i\to 0$ since $\overline{M_i}\to \overline{M}$ smoothly away from 
the origin.
Note also that 
\[
M_i\cap \{\sqrt{x^2+y^2}\ge r_i \}
\]
 is the union of $m$ graphs of functions
defined on
\[
  A(r_i,R) = \{ p\in \RR^2: r_i \le |p|\le R\}
\]
and that the tangent planes to those graphs all have slopes $\le L$.

By passing to a subsequence, we may assume that ${\lambda_i}r_i$ converges to a 
limit $r'$ in $[0,\infty]$.

If $r'<\infty$, the result follows immediately from the Annulus Lemma~\ref{annulus-lemma}.  (Apply the lemma to each
of the annular components of $M_i\cap \{\sqrt{x^2+y^2}\ge r_i\}$.)

Thus we may assume that $r'=\infty$:
\begin{equation}\label{eq:lambda-prime}
   \lambda_ir_i \to \infty.
\end{equation}

Let $M_i'$ be the result of dilating $M_i$ by $1/r_i$ about the origin.
By the Compactness Theorem~\ref{compactness-theorem}, the $M_i'$ converge smoothly (after passing to a
subsequence, and away from a discrete set) to a limit surface $M'$.
Since the slopes of the $M_i'$ are bounded by $L$ in the region $x^2+y^2\ge 1$,
the convergence $M_i'\to M'$ is smooth in the region $x^2+y^2>1$.

By definition of $r_i$, the the surfaces $M_i'$ and therefore  also $M'$ are contained in the set
\[
  \{ |z|\le \sqrt{x^2+y^2}\} \cup \{|z|\le 1\}.
\]
Applying the Annulus Lemma  
to each of the $m$-components of $M_i\cap\{x^2+y^2\ge r_i^2\}$, we see that
\begin{equation}\label{eq:horizontal}
 \text{$\slope(\Tan(M',p))\to 0$ as $|p|\to\infty$.}
\end{equation}
That is, the ends of $M'$ are horizontal.
Note that the number of ends of $M'$, counting multiplicities, is $m$.

The Blow-up Theorem~\ref{blow-up-theorem} asserts that one of the following must hold:
\begin{enumerate}
\item\label{good-case}
$M'$ is non-flat, complete, with finite total curvature, and the convergence $M_i'\to M'$ is 
smooth and multiplicity $1$
everywhere.
\item\label{bad-case} $M'$ is the a union of one or more parallel planes, possible with multiplicity.  
The convergence is smooth
except at isolated points.  
\end{enumerate}

In case (1), the surface $\lambda_i'M_i'$ converges smoothly to $\Tan(M',0)$ with multiplicity $1$
for every sequence $\lambda_i'\to\infty$; 
in particular, this holds for the  sequence $\lambda_i':=\lambda_i r_i$ 
(which tends to $\infty$ by~\eqref{eq:lambda-prime}).
But
\[
    \lambda_i'M_i' = (\lambda_i r_i)(1/r_i) M_i = \lambda_i M_i,
\]
so the $\lambda_iM_i$ converge with multiplicity $1$ to a plane (namely, $\Tan(M',0))$),
as desired.  This completes the proof in case (1).

Thus we may assume (2): that $M'$ is a union of parallel planes.
In this case, we know that the planes are horizontal since the ends are horizontal (see~\eqref{eq:horizontal}),
and that the number of planes (counting multiplicity) is $m$.

{\bf Case 2(a)}: $M'$ contains a plane not passing through the origin.
Then the plane that does pass through the origin has multiplicity $<m$.   

Let $\lambda_i'=\lambda_ir_i$, which tends to infinity by~\eqref{eq:lambda-prime}.
Then
\[
  \lambda_i'M_i' = \lambda_i r_i (1/r_i) M_i = \lambda_i M_i \to M^*.
\]
By the inductive hypothesis, $M^*$ either consists of a single multiplicity $1$ plane
or its ends are all horizontal.  Thus we are done in this case.

{\bf Case 2(b)}: $M'$ is the horizontal plane through $0$ with multiplicity $m$.

By the choice of $r_i$, there is a point $q_i=(x_i,y_i,z_i)$ in $M_i'$ such that
$x_i^2+y_i^2=1$ and such that
\begin{equation}\label{eq:witness}
          \slope(\Tan(M_i',q_i)) \ge L.
\end{equation}
In particular, the convergence $M_i'\to M'$ fails to be smooth at some point or points in
the circle $\{x^2+y^2=1, \, z=0\}$.

Let $U$ be the open ball of radius $1/2$ center at the origin.  Since (as we have just
shown) the convergence of $M_i'\setminus U$ to $M'\setminus U$ is not smooth
near the circle $\{x^2+y^2=1,  \, z=0\}$, it follows that
\[
   \limsup_i \TC(M_i'\setminus U) \ge 4\pi.
\]
In particular, we can assume (by passing to a subsequence) that
\[
     \TC(M_i'\setminus U) > 12
\]
for all $i$.  It follows that
\begin{align*}
\TC(M_i'\cap U) 
&= \TC(M_i') - \TC(M_i'\setminus U) 
\\
&< \TC(M_i') - 12
\\
&< N-12.
\end{align*}
Thus the proposition holds for the surfaces $M_i'$ and $M'$ by the inductive hypothesis.
Letting $\lambda_i'=\lambda_i r_i$ (which tends to $\infty$ by~\eqref{eq:lambda-prime}), we have
\[
  \lambda_i' M_i' = (\lambda_i r_i ) ((1/r_i)M_i) = \lambda_i M_i \to M^*.
\]
By the inductive hypothesis, $M^*$ satisfies the conclusions of the theorem.
\end{proof}

\section{Appendix}

\begin{theorem}\label{removal-theorem}
Let $\BB(p,r)$ be the open ball of radius $r$ centered at $p\in \RR^3$.
Suppose that $g$ is a smooth Riemannian metric on $\BB(p,r)$,
and that $M$ is a properly embedded, $g$-minimal surface in $\BB(p,r)\setminus \{p\}$
with finite total curvature and finite area, and that $p\in \overline{M}$.
Then $M\cup \{p\}$ is a smoothly embedded, $g$-minimal surface.
\end{theorem}

\begin{proof}
We may assume that $p=0$.
It follows from the first variation formula that $M\cap \BB(0,\rho)$ has finite area for every $\rho<r$.
(This is true for any bounded mean curvature variety in $\BB(0,r)\setminus\{p\}$
by the first variation formula.  See, for example, \cite{gulliver-removability}*{lemma 1}.)
Thus by replacing $\BB(0,r)$ by a smaller ball, we can assume that the the area of $M$ is finite
and that the total curvature of $M$ is less than $4\pi$.

Let $\lambda_i\to \infty$.  By the Compactness Theorem~\ref{compactness-theorem},
after passing to a subsequence, $\lambda_iM_i$ converges
smoothly on compact subsets of $\RR^3\setminus \{0\}$ to a $g(0)$-minimal surface $M'$.
Note that $M'$ is a $g(0)$-minimal cone (it is a tangent cone to $M$ at $0$) and is smooth without transverse
self-intersections, so it is a plane.

It follows that there is an $\eps>0$ such that the
function
\[ 
x\in M\cap \BB(0,\eps) \mapsto |x|
\]
has no critical points, which implies that  $M\cap \BB(0,\eps)$ is a union of surfaces $D_1,\dots, D_k$, each of which is topologically
a punctured disk.

By a theorem of Gulliver~\cite{gulliver-removability}, each $D_i\cap \{0\}$ is a (possibly branched) minimal disk.
However, since $D_i$ has no transverse self-intersections, $D_i\cap \{0\}$ is smoothly embedded.
By the strong maximum principle, there is only one such disk.
\end{proof}

\begin{remark}
A different proof (not using Gulliver's Theorem) is given in \cite{white-curvature-estimates}*{Theorem~2}.
\end{remark}

\begin{remark}
Whether the finite total curvature assumption is necessary is a very interesting  open
problem in minimal surface theory.
The theorem remains true if that assumption is replaced by the assumption that $M$ is stable~\cite{gulliver-lawson},
or by the assumption that $M$ has finite Euler characteristic~\cite{choi-schoen}*{Proposition~1}.
It also remains true if that assumption is replaced by the assumption that $M$ has finite genus.
(Using monotonicity and lower bounds on density, one can show that for sufficiently small $\eps>0$, 
the surface $M\cap \BB(p,\eps)$ is a union of finitely many surfaces homeomorphic to punctured disks,
to which one can then apply Gulliver's theorem~\cite{gulliver-removability}.)
\end{remark}

\begin{theorem}\label{sobolev-theorem}
Let $M$ be a smooth, two-dimensional Riemannian manifold without boundary, $f$ be a smooth function on $M$,
and $p$ be a point in the interior of $M$.  Then $C_c^\infty(M\setminus\{p\})$ is dense with
respect to the $H^1$ norm in $C_c(M)$.
\end{theorem}

Equivalently, $H^1_0(M)=H^1_0(M\setminus\{p\})$.

\begin{corollary}\label{cutoff-corollary}
If there is a $u\in C^\infty_c(M)$ such that 
\[
  \int ( |Du|^2 + f|u|^2) \,dA < 0,
\]
then there is a $u\in C^\infty_c(M\setminus \{p\})$ satisfying the same inequality.
\end{corollary}

\begin{proof}[Proof of Theorem~\ref{sobolev-theorem}]
The theorem is essentially local, and independent of the choice of metric.
Thus we can assume that $M$ is the open unit disk $D$ in $\RR^2$ with the Euclidean metric and
that $p=0$.  Given $u\in C^\infty_c(D)$ and $0<\eps<1$, define $u_\eps:D\to\RR$ by
\[
u_\eps(z) 
= 
\begin{cases}
u(z) 
&\text{if $|z|\ge \eps$}, 
\\
\frac{ \ln |z| - \ln(\eps^2)}{\ln(\eps) - \ln (\eps^2)} u(z)  
&\text{if $\eps^2 \le |z| \le \eps$, and}
\\
0 
&\text{if $|z|\le \eps^2$}.
\end{cases}
\]
One readily checks that $u_\eps$ converges to $u$ in $H^1$ as $\eps\to 0$.
Of course $u_\eps$ is not smooth, but it is Lipschitz and compactly supported in $C^\infty_c(D\setminus\{0\})$,
so we can mollify to approximate it arbitrary well in $H^1$ by a function in $C^\infty_c(D\setminus\{0\}$.)
\end{proof}

\nocite{pedrosa-ritore}
\nocite{hoffman-wei}
\newcommand{\hide}[1]{}

\end{document}